\documentclass[preprint,12pt]{elsarticle}

\usepackage{amssymb}
\usepackage{amsthm}

\newcommand{\sysn}{\left\{\begin{array}{rcl}}
\newcommand{\sysk}{\end{array}\right.}

\newtheorem{theorem}{Theorem}[section]

\theoremstyle{example}
\newtheorem{example}[theorem]{Example}

\theoremstyle{definition}
\newtheorem{definition}[theorem]{Definition}

\newtheorem{corollary}[theorem]{Corollary}

\journal{Fundamenta Mathematicae}

\begin{document}

\begin{frontmatter}



\title{Variations of selective separability and  tightness in
function spaces with set open topologies}


\author[label1]{Alexander V. Osipov}

\ead{OAB@list.ru}

\tnotetext[label1]{ This work was supported by Act 211 Government
of the Russian Federation, contract ¹ 02.A03.21.0006.}

\author[label2]{Selma \"{O}z\c{c}a\u{g}}

 \ead{sozcag@hacettepe.edu.tr}

\address[label1]{Ural Federal
 University, Krasovskii Institute of Mathematics and Mechanics,
Ekaterinburg, Russia}

\address[label2]{Department of Mathematics, Hacettepe  University, Beytepe, Ankara, Turkey}

\begin{abstract}
We study tightness properties and selective versions of
separability in bitopological function spaces endowed with
set-open topologies.
\end{abstract}


\end{frontmatter}



\section{Introduction}

\smallskip
In this paper we are mainly concerned with selective versions of
separability in bitopological  function spaces endowed with two
homogenous set-open topologies.

\smallskip
Variations of separability, stronger forms, weaker forms,
functional separability and similar properties have been
intensively studied by many mathematicians in the last several
decades. The selection versions of separability has recently
gained a particular attention and as a consequences many interesting results were
obtained.

\smallskip
Although the definition of the selection principles was given by Scheepers in 1996
the initial studies of the theory was based on the papers  by
Menger, Hurewicz, Rothberger and Sierpinski in
1920-1930, see~{\cite{hur,menger,roth}}.

\smallskip
Many topological properties can be  defined or characterized in terms
 of the following two classical selection principles given in a general form in ~{\cite{sc}}  as follows:

Let $\mathcal{A}$ and $\mathcal{B}$ be sets consisting of
families of subsets of an infinite set $X$. Then:

$S_{1}(\mathcal{A},\mathcal{B})$: for
each sequence $(A_{n}: n\in \mathbb{N})$ of elements of
$\mathcal{A}$ there is a sequence $(b_{n}: n\in\mathbb{N})$ such
that for each $n$, $b_{n}\in A_{n}$, and $\{b_{n}: n\in\mathbb{N}
\}$ is an element of $\mathcal{B}$.

$S_{fin}(\mathcal{A},\mathcal{B})$:
for each sequence $(A_{n}: n\in \mathbb{N})$ of elements of
$\mathcal{A}$ there is a sequence $(B_{n}: n\in\mathbb{N})$ of
finite sets such that for each $n$, $B_{n}\subseteq A_{n}$, and
$\bigcup_{n\in\mathbb{N}}B_{n}\in\mathcal{B}$.











\smallskip
These selection principles denoted by $S_{fin}(\mathcal{O},
\mathcal{O})$ and $S_{1}( \mathcal{O}, \mathcal{O})$  are called
Menger and Rothberger property where $\mathcal{O}$ is the family
of open covers of a topological space.

\smallskip
For the topological space $X$, let $\mathcal{D}$ denote the family
of dense subspaces of $X$. The selection principles
$S_{fin}(\mathcal{D}, \mathcal{D})$ and $S_{1}( \mathcal{D},
\mathcal{D})$ were introduced by Scheepers   in~{\cite{sch}}
and recently gained a great attention, see~{\cite{bbm,bbmt,bbms,gs}}.

\smallskip
In~{\cite{bbm}} the selection properties $S_{fin}(\mathcal{D},
\mathcal{D})$, $S_{1}(\mathcal{D}, \mathcal{D})$ and
$S_{1}(\mathcal{D}, \mathcal{D}^{gp})$ are called $M$-separability
(also called selective separability), $R$-separability and
$GN$-separability, respectively, while a bit modified property
$S_{fin}(\mathcal{D}, \mathcal{D}^{gp})$ is called
$H$-separability where "M-", "R-" and "H-" represent well known
Menger, Rothberger and Hurewicz properties.

\smallskip
It should be noted that very recently Tsaban and his co-authors
in~{\cite{tsaban2}} studied all properties
$S(\mathcal{A},\mathcal{B})$ for $S\in \{S_1, S_{fin}\}$ and
$\mathcal A,\mathcal B$ are combinations of open covers, dense
open families and dense sets.

\smallskip
The selection principle theory is firstly considered in bitopological spaces by
Ko\v cinac and \"{O}z\c{c}a\u{g} in  ~{\cite{kooz,kooz1}} and they
carried out a systematic study on selection principles  mainly
selective versions of separability in bitopological spaces,
particularly in the space $C(X)$ of all continuous real-valued
functions defined on a Tychonoff space~$X$, where $C(X)$ is
endowed with the topology $\tau_p$ of pointwise topology and the
compact-open topology $\tau_k$.

\smallskip
 In the following we investigate some properties of bitopological selective versions of separability
 in function spaces and the set-open topologies will be used as a main tool.

\smallskip
The set open topology on a family $\lambda$ of nonempty subsets of
the set $X$ is a generalization of the compact open topology and
of the topology of pointwise convergence. This notion was first
introduced by Arens and Dugundji in ~{\cite{ad}} and was widely
investigated by Osipov in ~{\cite{os2,os1,os3}}. In the next
section we recall some facts on the set open topologies.

\smallskip
For the background material on selection principles we refer
 to the survey papers ~{\cite{ko1,sc1,tsaban}}, for the undefined notions in function spaces,
 see~{\cite{arh}} . We will follow~{\cite{enge}} for topological terminology and notations.

\section{Main definitions and notation}

Recall that a subset $A$ of a space $X$ is called
 {\it $C$-compact subset of $X$} if, for any real-valued function~$f$ continuous on $X$,
 the set $f(A)$ is compact in ~${\mathbb{R}}$.

Let $X$ be a topological space. Then:

 $\Psi$ denotes the collection of all $\pi$-networks of closed $C$-compact subsets
 of the set~$X$ such that it is
closed under $C$-compact subsets of the set $X$ of its elements.

 A family $\lambda$ of $C$-compact subsets of $X$ is said to be
 closed under (hereditary with respect to) $C$-compact subsets if it satisfies
 the following condition: whenever $A\in \lambda$ and $B$ is
 a $C$-compact (in $X$) subset of $A$, then $B\in \lambda$ also.

Note that $p\in \Psi$ and $k\in \Psi$ where $p$ and $k$ are sets
all finite and compact subsets of $X$.

 We use the
following notation for various topological spaces
 with the underlying set $C(X)$:

 $C_{\lambda}(X)$ for the $\lambda$-open topology.

 The element of the standard subbase of the $\lambda$-open
 topology:

 $[F,\,U]=\{f\in C(X):\ f(F)\subseteq U\}$ where $F\in \lambda$ and $U$ is a open subset of $\mathbb{R}$.

Given a family $\lambda$ of non-empty subsets of $X$, let
$\lambda(C)=\{A\in \lambda$ :  for every $C$-compact subset $B$ of
the space $X$ with $B\subset A$, the set $[B,U]$ is open in
$C_{\lambda}(X)$ for any open set $U$ of the space $\mathbb{R}
\}$.

 Let $\lambda_m$ denote the maximal with respect to inclusion
 family, provided that $C_{\lambda_m}(X)=C_{\lambda}(X)$. Note that a family
 $\lambda_m$ is unique for each family $\lambda$.

Interest in studying the $\lambda$-open topology generated by a
Theorem 3.3 in ~{\cite{os2}} which characterizes some
topological-algebraic properties of the set-open topology.

The following theorem is a corollary of Theorem 3.3 in~{\cite{os2}}.

\begin{theorem}\label{th1.1} For a space $X$, the following statements are
equivalent.

\begin{enumerate}

\item  $C_{\lambda}(X)$ is a paratopological group.

\item  $C_{\lambda}(X)$ is a topological group.

\item  $C_{\lambda}(X)$ is a topological vector space.

\item $C_{\lambda}(X)$ is a locally convex topological vector
space.

\item $C_{\lambda}(X)$ is a topological ring.

\item $C_{\lambda}(X)$ is a topological algebra.

\item  $\lambda$ is a family of\, $C$-compact sets and
$\lambda=\lambda(C)$.

\item $\lambda_m$ is a family of\, $C$-compact sets and it is
hereditary with respect to $C$-compact subsets.

\end{enumerate}

\end{theorem}

So without loss of generality we can assume  that $C_{\lambda}(X)$
is a paratopological group (TVS , locally convex TVS) under the
usual operations of addition and multiplication (and
multiplication by scalars) iff $\lambda_{m}\in \Psi$.

Further, throughout the article, we assume $\lambda=\lambda_{m}\in
\Psi$.

In particular, if $\lambda=p$ ($\lambda=k$) then
$C_{\lambda}(X)=C_p(X)$ ($C_{\lambda}(X)=C_{k}(X)$), i.e.
$\lambda$-open topology coincide with the topology of pointwise
convergence (the compact-open topology).

Since $C_{\lambda}(X)$ is homogenous space we may always consider
the point $\textbf{0}$ when studying local properties of this
space.

We use the symbol $\Omega_{x}$ to denote the set $\{A\subset X:
x\in Cl_{\tau}(A)\setminus A \}$, where $(X,\tau)$ is a
topological space and $x\in X$.

 Since we are considering two topologies of $C(X)$ we shall
use the symbol $(\Omega_{0})^{\lambda}$ to denote $\Omega_{0}$ in
the space $C_{\lambda}(X)$ and the symbol $(\Omega_{0})^{\mu}$ to
denote $\Omega_{0}$ in the space $C_{\mu}(X)$ where
$\lambda,\mu\in \Psi$.

We will denote by $\tau_{\lambda}$ the $\lambda$-open topology on
$C(X)$ for $\lambda\in \Psi$.

Before we proceed let us recall some definitions, notations and
terminology which were used in~{\cite{kooz,kooz1}}.

Throughout this paper $(X,\tau_1, \tau_2)$ will be a bitopological
space (shortly bispace), i.e. the set $X$ endowed with two
topologies $\tau_1$ and $\tau_2$. For a subset $A$ of $X$,
$Cl_{i}$ will denote the closure of $A$ in $(X,\tau_i)$, $i=1,2$.

A subset $A$ of $X$ is bidense (shortly $d$-dense) in $(X,\tau_1,
\tau_2)$ if $A$ is dense in both $(X,\tau_1)$ and $(X,\tau_2)$.
$X$ is $d$-separable if there is a countable set $A$ which is
$d$-dense in $(X,\tau_1, \tau_2)$.

Denote by $\mathcal{D}_1$ and $\mathcal{D}_2$ the collections of
all dense subsets of $(X,\tau_1)$ and $(X,\tau_2)$, respectively.
We say that $X$ is:

$M_{\tau_i,\tau_j}$-separable $(i,j=1,2; i\neq j)$, if for each
sequence $(D_n: n\in \mathbb{N})$ of elements of $\mathcal{D}_i$
there are finite sets $F_n\subset D_n$, $n\in \mathbb{N}$, such
that $\bigcup_{n\in \mathbb{N}} F_{n}\in \mathcal{D}_j$, i.e. if
$S_{fin}(\mathcal{D}_i,\mathcal{D}_j)$ hold;

$R_{\tau_i,\tau_j}$-separable if
$S_{1}(\mathcal{D}_i,\mathcal{D}_j)$ hold;

$H_{\tau_i,\tau_j}$-separable if for each sequence $(D_n: n\in
\mathbb{N})$ of elements of $\mathcal{D}_i$ there are finite sets
$F_n\subset D_n$, $n\in \mathbb{N}$, such that each $\tau_j$-open
subset of $X$ intersects $F_n$ for all but finitely many $n$;

$GN_{\tau_i,\tau_j}$-separable if
$S_{1}(\mathcal{D}_i,\mathcal{D}^{gp}_j)$ hold.

Here $\mathcal{D}^{gp}_j$ is the collection of groupable dense
subsets of a space; a countable dense subset $D$ of a space $Z$ is
groupable if $D=\bigcup_{n\in \mathbb{N}} A_{n}$, each $A_n$
finite and each open set $U$ in $Z$ intersects all but finitely
many $A_n$ ~{\cite{mkm}}.

In case $\tau_1=\tau_2=\tau$, then these definitions coincide
with definitions of corresponding topological selective versions
of separability of $(X,\tau)$.

As mentioned in~{\cite{kooz}} we have the implications,
$GN_{\tau_i,\tau_j}$-separable $\Longrightarrow$
$R_{\tau_i,\tau_j}$-separable $\Longrightarrow$
$M_{\tau_i,\tau_j}$-separable, and $H_{\tau_i,\tau_j}$-separable
$\Longrightarrow$ $M_{\tau_i,\tau_j}$-separable.

\smallskip
The remaining notations can be found in~{\cite{enge,kooz,kooz1}}.

\section{The tightness-type properties}

In this section we give  some results on bitopological versions of
the tightness properties and its variations in function bispaces.
We also combine these results with bitopological selective
separability properties.

In analogy to the $(\tau_i,\tau_j)$-tightness in ~{\cite{kooz}}
$(\tau_{\lambda},\tau_{\mu})$-tightness was introduced by
replacing $\tau_i$ and $\tau_j$ with $\tau_{\lambda}$ and
$\tau_{\mu}$ topologies respectively.

The $(\tau_{\lambda},\tau_{\mu})$-tightness, $\lambda, \mu\in
\Psi$, of a bispace $(C(X),\tau_{\lambda},\tau_{\mu})$ is the
least infinite cardinal $\kappa$ such that whenever $A\subseteq
C(X)$ and $f\in Cl_{\tau_{\lambda}}(A)$, there is $B\subseteq A$
such that $\mid B \mid \leq \kappa$ and $f\in Cl_{\tau_{\mu}}(B)$.
\medskip

We recall that a subset $A$ of $X$ is called co-zero (or a
functional open) set if $X\setminus A$ is a zero set. We mean by a
zero set, a subset of $X$ that is complete preimage of zero for
certain function from $C(X)$.

\begin{definition}

A co-zero (functional open) family  $\mathbb{U}$ of $X$ is called
a $\lambda$-$f$-cover  if $X$ is not a member of $\mathbb{U}$ and
for each $A\in \lambda$ there is a $U\in \mathbb{U}$ such that
$A\subseteq U$.

\end{definition}
Note that $\lambda$-$f$-cover is the cover of $\bigcup \lambda$,
but it can not be the cover of $X$.
The symbol $\Lambda(\lambda)$ denotes  the collection of all
$\lambda$-$f$-covers for the family $\lambda$.

\begin{definition}
A space $X$ is $\lambda$-$\mu$-Lindel$\ddot{o}$f if for each
 $\mathcal{U}\in \Lambda(\lambda)$ there is a
$\mathcal{V}\subseteq \mathcal{U}$ such that $\mathcal{V}$ is
countable and $\mathcal{V}\in \Lambda(\mu)$.
\end{definition}

If $\lambda=\mu$ then we shall write it simply
$\lambda$-Lindel$\ddot{o}$f. Note that if $\lambda$=$k$ then
$\lambda$-Lindel$\ddot{o}$f is $k$-Lindel$\ddot{o}$f.

\begin{theorem} For a space $X$, $\lambda$, $\mu \in \Psi$ and $\mu\subseteq
\lambda$, the bispace $(C(X),\tau_{\lambda},\tau_{\mu})$ has
countable  $(\tau_{\lambda},\tau_{\mu})$-tightness if and only if
$X$ is $\lambda$-$\mu$-Lindel$\ddot{o}$f.
\end{theorem}

\begin{proof}
$(\Rightarrow)$. Let $(C(X),\tau_{\lambda},\tau_{\mu})$ be
countable $(\tau_{\lambda},\tau_{\mu})$-tightness and
$\mathcal{U}\in \Lambda(\lambda)$. For each pair of a element $K$
of $\lambda$ and $U\in \mathcal{U}$, $K\subseteq U$ let $f_{K,U}$
be any continuous function from $X$ to $[0,1]$ such that
$f_{K,U}(K)\subseteq \{0\}$ and $f_{K,U}(X\setminus U)\subseteq
\{1\}$. Let $A=\{f_{K,U}: K\in \lambda, K\subseteq U\in
\mathcal{U} \}$. Then $\textbf{0}$ belongs to the closure of $A$
with respect to the $\tau_{\lambda}$ topology. Since
$(C(X),\tau_{\lambda},\tau_{\mu})$ has countable
$(\tau_{\lambda},\tau_{\mu})$-tightness there is a countable set
$B=\{ f_{K_n,U_n} : n\in \mathbb{N} \}$ such that $\textbf{0}$
belongs to the closure of $B$ with respect to the $\tau_{\mu}$
topology. We claim that $\{U_n : n\in \mathbb{N} \}\in
\Lambda(\mu)$. Let $F\in \mu$. From the fact that $\textbf{0}$
belongs to the closure of $B$ with respect to the $\tau_{\mu}$
topology it follows that there is an $i\in \mathbb{N}$ such that
$[F,(-1,1)]$ contains the function $f_{K_{i},U_{i}}$. Then
$F\subseteq U_{i}$. Overwise for some $x\in F$ one has $x\notin
U_{i}$ so that $f_{K_{i},U_{i}}(x)=1$, contradicting
$f_{K_{i},U_{i}}\in[F,(-1,1)]$.

$(\Leftarrow)$. Let $A$ be a set of $C(X)\setminus \{\textbf{0}\}$
the closure of which contain $\textbf{0}$, with respect to the
$\tau_{\lambda}$ topology. If $\{X\}\in \lambda$ ($X$ is
pseudocompact) then the $\tau_{\lambda}$ topology coincides with
the $C$-compact-open topology, so $C_{\lambda}(X)$ is metrizable
(Theorem 2.2 in ~{\cite{os3}}), thus first countable, which means
that we can find a sequence $(a_{n}: n\in \mathbb{N})$, converging
uniformly to $\textbf{0}$ so there is nothing to be proved.

 Let $\{X\}\notin \lambda$. For each $n\in \mathbb{N}$ and every set
$K\in \lambda$ the neighborhood $[K,(- 1/n, 1/n)]$ of $\textbf{0}$
intersects $A$, so there exists a continuous function $f_{K,n}\in
A$ such that $|f_{K,n}(x)|<1/n$ for each $x\in K$. Since $f_{K,n}$
is a continuous function there is a co-zero set $U_{K,n}$ such
that $f_{K,n}(U_{K,n})\subseteq (-1/n, 1/n)$. Let $\mathcal
U_{n}=\{U_{K,n}: K\in \lambda \}$.

 As for any subset $K\in \lambda$ we have $K\neq X$, it can easily
 be achieved that none of the sets $U_{K,n}$ above coincides with
 $X$. So for each $n$, $\mathcal U_{n}\in \Lambda(\lambda)$. Each $\mathcal
 U_{n}$ has countable $\mu$-$f$-cover $\mathcal{V}_{n}\subseteq
 \mathcal{U}_{n}$. Define $B=\{f_{K,n} : n\in \mathbb{N},
 U_{K,n}\in \mathcal{V}_{n} \}$. It is evident that $B\subseteq A$,
 $|B|\leq\aleph_{0}$, and $\textbf{0}$ is the closure of $B$ with respect to the $\tau_{\mu}$
topology. Therefore the bispace $(C(X),\tau_{\lambda},\tau_{\mu})$
has countable  $(\tau_{\lambda},\tau_{\mu})$-tightness.
\end{proof}

\begin{corollary}
The space $C_{\lambda}(X)$ has countable tightness if and only if
$X$ is $\lambda$-Lindel$\ddot{o}$f.
\end{corollary}

\begin{definition}
A space $X$ has countable $(\tau_{\lambda},\tau_{\mu})$-fan
tightness if for each $x\in X$ and each sequence $(A_{n}: n\in
\mathbb{N})$ of elements of $(\Omega_{x})^{\lambda}$ there exists
a sequence $(B_{n}: n\in \mathbb{N})$ of finite sets such that,
for each $n$, $B_{n}\subseteq A_{n}$ and $x\in
cl_{\tau_{\mu}}(\bigcup_{n\in \mathbb{N}}B_{n})$, i.e. if
$S_{fin}((\Omega_{x})^{\lambda},(\Omega_{x})^{\mu})$ holds for
each $x\in X$.
\end{definition}

\begin{definition}
A space $X$ has countable $(\tau_{\lambda},\tau_{\mu})$-strong fan
tightness if for each $x\in X$
$S_{1}((\Omega_{x})^{\lambda},(\Omega_{x})^{\mu}))$ holds.
\end{definition}

The proof of the next theorem follows much the similar lines as in the proof
of the theorem 2.4 in~{\cite{panpa}}.

\begin{theorem}
\label{th4.9}
 Let $X$ be a Tychonoff space,
$\lambda$, $\mu \in \Psi$ and $\mu\subseteq \lambda$. Then the
following are equivalent:

\begin{enumerate}

\item  $C(X)$ satisfies
$S_{1}((\Omega_{0})^{\lambda},(\Omega_{0})^{\mu}))$;

\item  $X$ has property $S_{1}(\Lambda(\lambda),\Lambda(\mu))$.

\end{enumerate}

\end{theorem}

\begin{proof}

$(1)\Rightarrow(2)$. Let $(\mathcal{U}_{n}: n\in \mathbb{N})$ be a
sequence of $\lambda$-$f$ covers of $X$. For each pair $A\in
\lambda$ and a co-zero set $U$ such that  $A\subseteq U$ let
$f_{A,U}:X \rightarrow [0,1]$ be a continuous function with
$f_{A,U}(A)\subseteq \{0\}$ and $f_{A,U}(X\setminus U)\subseteq
\{1\}$.

Now let  $B_{n}=\{ f_{A,U} : A\in \lambda, A\subseteq U\in
\mathcal{U}_{n}\}$. Now we claim that $\textbf{0}$  is in the
closure of each $B_n$, with respect to $\lambda$-open topology.

Indeed at first we have $\textbf{0}\notin B_n$ and secondly let
$\textbf{0}\in [A,V]$ where $A\in \lambda$ and $V$ is open subset
of $\mathbb R$.  There exists $U\in \mathcal U_n$ with $A\subseteq
U $. For the functions $f_{A,U}\in B_{n}$ $f_{A,U}\in [A, V]$ and
$f_{A,U}(A)=\{0\}\subseteq V$ hence $\textbf{0}$  is in the
closure of $B_n$ with respect to $\lambda$-open topology, which
means $B_n\in (\Omega_0)^\lambda$.

Now since $C(X)$ has countable $(\tau_\lambda,\tau_\mu)$ strong
fan tightness there is a sequence $(f_{A_{n},U_{n}}: n\in
\mathbb{N})$  $\forall n\in \mathbb N$ with  $A_{n}\in \lambda$,
$U_{n}\in \mathcal{U}_{n}$ such that $\textbf{0}$ belongs to the
closure of $\{f_{A_{n},U_{n}}: n\in \mathbb{N}\}$ with respect to
the $\mu$-open topology. To conclude the first part of the proof
now we claim that $\{U_{n}: n\in \mathbb{N}\}\in \Lambda(\mu)$.
Let $K\in \mu$. There is an $j\in \mathbb{N}$ such that
$[K,(-1,1)]$ contains the function $f_{A_{j},U{j}}$. Clearly
$K\subseteq U_{j}$.

\smallskip
$(2)\Rightarrow(1)$. Let $(B_{m}: m\in \mathbb{N})\in
{\Omega_0}^\lambda$.  If $\{X\}\in \lambda$ it is clear. Now let
$\{X\}\notin \lambda$. We set $B_{n,m}:=B_{j(n,m)}$ for the
bijection $j: \mathbb{N}^2 \mapsto \mathbb{N}$. For each $A\in
\lambda$ the neighborhood  $[A,(- 1/n, 1/n)]$ of $\textbf{0}$
intersects $B_{m,n}$, which means that there exists a continuous
function   $f_{A,m,n}\in B_{m,n}$ such that $|f_{A,m,n}(x)|<1/n$
for each $x\in A$  and    $f_{A,m,n}(U_{A,m,n})\subseteq (-1/n,
1/n)$ for the co-zero set $U_{A,m,n}$.

Now we set  $\mathcal U_{m,n}={U_{A,m,n}, A\in \lambda}$. For any
set $A\in \lambda$, $A\neq X$,none of the sets $U_{A,m,n}$ equals
to $X$ gives us. So for each $m$ and $n$, $\mathcal U_{m,n}\in
 \Lambda(\lambda)$. To each sequence $(\mathcal U_{m,n} : m\in
 \mathbb{N})$ apply the fact that $X$ is an
 $S_{1}(\Lambda(\lambda),\Lambda(\mu))$-space. We can easily
 obtain a sequence $(U_{A_{m,n}}: m\in \mathbb{N})$, such that
  $\{U_{A_{m,n}}: m\in \mathbb{N}\}\in \Lambda(\mu)$. Now define
  $D=\{f_{A_{m,n}} : n\in \mathbb{N},
 U_{A,m,n}\in\{U_{A_{m,n}}: m\in \mathbb{N}\} \}$. It is evident that $D\subseteq B_n$,
 $|D|\leq\aleph_{0}$, and $\textbf{0}$ is the closure of $D$ with respect to the $\tau_{\mu}$
topology. Therefore the bispace $(C(X),\tau_{\lambda},\tau_{\mu})$
satisfies $S_{1}((\Omega_{0})^{\lambda},(\Omega_{0})^{\mu}))$.
\end{proof}

\begin{corollary} For a Tychonoff space $X$ the following are equivalent:

\begin{enumerate}

\item  $C_{\lambda}(X)$ has countable strong fan tightness;

\item  $X$ has property
$S_{1}(\Lambda(\lambda),\Lambda(\lambda))$.

\end{enumerate}

\end{corollary}
In a similar way one can show that

\begin{theorem}
\label{thg}
 Let $X$ be a Tychonoff space,
$\lambda$, $\mu \in \Psi$ and $\mu\subseteq \lambda$. Then the
following are equivalent:

\begin{enumerate}

\item  $C(X)$ satisfies
$S_{fin}((\Omega_{0})^{\lambda},(\Omega_{0})^{\mu}))$;

\item  $X$ has property $S_{fin}(\Lambda(\lambda),\Lambda(\mu))$.

\end{enumerate}

\end{theorem}

\begin{corollary}
The space $C_{\lambda}(X)$ has countable fan tightness if and only
if $X$ has property $S_{fin}(\Lambda(\lambda),\Lambda(\lambda))$.
\end{corollary}

\medskip
The $T$ -tightness $T(X)$ of a space $X$ is the smallest infinite
cardinal $\tau$ such that if $\{F_{\alpha}: \alpha<\kappa\}$ is an
increasing family of closed subsets and $cf(\kappa)> \tau$, then
$\bigcup\{F_{\alpha}:\alpha<\kappa\}$ is closed in $X$. This
definition was introduced in  ~{\cite{uch}} by Juh$\acute{a}$sz .
Since the family $\{F_{\alpha}: \alpha<\kappa\}$ is increasing and
$cf(\kappa)$ is regular, we may say that the $T$ -tightness $T(X)$
is the smallest infinite cardinal $\tau$ such that if
$\{F_{\alpha}: \alpha<\kappa\}$ is an increasing family of closed
subsets and $\kappa$ is a regular cardinal greater than $\tau$,
then $\bigcup \{F_{\alpha}: \alpha<\kappa\}$ is closed in $X$. In
\cite{masa} the T-tightness of function spaces $C_p X)$ was
investigated and in \cite{ko} it was considered for $C_k (X)$.

\smallskip
The bitopological version of this notion  was
introduced in ~{\cite{kooz1}}.

\noindent
A bispace $(X,\tau_{\mu},\tau_{\lambda})$ has countable
$(\tau_{\mu},\tau_{\lambda})$ - $T$ - tightness, if for each
uncountable regular cardinal $\kappa$ and each increasing sequence
$(A_{\alpha}: \alpha<\kappa)$ of closed subsets of
$(X,\tau_{\mu})$ the set $\bigcup\{A_{\alpha}: \alpha<\kappa\}$ is
closed in $(X,\tau_{\lambda})$.

\begin{theorem}\label{th1} Let $X$ be a Tychonoff space,
$\lambda$, $\mu \in \Psi$ and $\mu\subseteq \lambda$. Then (1)
implies (2):

\begin{enumerate}

\item   $(C(X),\tau_{\mu},\tau_{\lambda})$  has  countable
$(\tau_{\lambda},\tau_{\mu})$ - $T$ - tightness;

\item  for each regular cardinal $\kappa$ and each increasing
sequence $\{\mathcal{U}_{\alpha}: \alpha<\rho \}$ of family of
cozero subsets of $X$ such that $\bigcup_{\alpha<\kappa}
\mathcal{U}_{\alpha}$ is a $\lambda$-$f$-cover of $X$ there is a
$\beta<\kappa$ so that $\mathcal{U}_{\beta}$ is a $\mu$-$f$-cover
of $X$.

\end{enumerate}

\end{theorem}

\begin{proof}

$(1)\Rightarrow(2)$. Since a bispace
$(C(X),\tau_{\mu},\tau_{\lambda})$ has countable
$(\tau_{\lambda},\tau_{\mu})$-$T$-tightness
  for each uncountable regular cardinal
$\kappa$ and each increasing sequence $ \{ A_{\alpha} : \alpha<
\kappa \}$ of closed subsets of $(C(X),\tau_{\lambda})$ the set
$\bigcup \{ A_{\alpha} : \alpha< \kappa \}$ is closed in
$(C(X),\tau_{\mu})$. Let for regular cardinal $\kappa$ and
increasing sequence $\{\mathcal{U}_{\alpha}: \alpha<\rho \}$ of
family of cozero subsets of $X$ the family
$\bigcup_{\alpha<\kappa} \mathcal{U}_{\alpha}$ be a
$\lambda$-$f$-cover of $X$.  Now we set $\mathcal{U}_{\alpha,K}:=\{ \mathcal{U}\in
\mathcal{U}_{\alpha} : K\subset \mathcal{U} \}$  for each  $\alpha<\kappa$ and
$K\in \lambda$. For each
$\mathcal{U}\in \mathcal{U}_{\alpha,K}$ let $f_{K,\mathcal{U}}$ be
a continuous function from $X$ into $[0,1]$ such that
$f_{K,\mathcal{U}}(K)=\{0\}$ and $f_{K,\mathcal{U}}(X\setminus
\mathcal{U})=\{1\}$. Now consider the set  $A_{\alpha}=\{f_{K,\mathcal{U}} :
\mathcal{U}\in \mathcal{U}_{\alpha,K} \}$, $\alpha< \kappa$.

By (1) we observe that the set   $A=\bigcup_{\alpha< \kappa}
cl_{\tau_{\lambda}}(A_{\alpha})$ is closed in $(C(X),\tau_{\mu})$.

Let $<\textbf{0},K,\varepsilon>:=[K,(-\varepsilon, \varepsilon)]$
be a standard basic $\tau_{\lambda}$-neighborhood of $\textbf{0}$.
There exist $\alpha< \kappa$ and $\mathcal{U}\in
\mathcal{U}_{\alpha}$ with $K\subset \mathcal{U}$. Then
$\mathcal{U}\in \mathcal{U}_{\alpha,K}$, hence by construction
there is $f\in A_{\alpha}\bigcap <\textbf{0},K,\varepsilon>$.
Therefore, each $\tau_{\lambda}$-neighborhood of $\textbf{0}$
intersects some $A_{\alpha}$, $\alpha<\kappa$, i.e. $\textbf{0}$
belongs to the $\tau_{\lambda}$-closure of the set
$\bigcup_{\alpha< \kappa} A_{\alpha}$ which is actually the set
$A$. It follows that there is $\beta< \kappa$ with $\textbf{0}\in
cl_{\tau_{\mu}}(A_{\beta})$. We claim that the corresponding
family $\mathcal{U}_{\beta}$ is an $\mu$-$f$-cover of $X$.

Let $F\in \mu$. Then the $\tau_{\mu}$-neighborhood
$<\textbf{0},F,1>$ of $\textbf{0}$ intersects $A_{\beta}$; let
$f_{F,\mathcal{U}}\in A_{\beta}\bigcap <\textbf{0},F,1>$. Then
$f_{F,\mathcal{U}}(X\setminus\mathcal{U})=1$ and thus
$F\subset\mathcal{U}\in\mathcal{U}_{\beta}$. Then $\mathcal{U}_{\beta}$ is an
$\mu$-$f$-cover of $X$.

\end{proof}
The following example show that the condition (2) may not involve
the condition (1).

\bigskip

{\bf Example 1.}  Let $X=\omega_1+1$ be the space $\{\alpha :
\alpha\leq \omega_1 \}$ with the order topology, $\mu=p$,
$\lambda=k$. Consider bitopological space
$(C(X),\tau_{p},\tau_{k})$. Note that $C_{p}(X)$ has countable
$T$-tightness (Theorem 2.3 in ~{\cite{masa}})  and hence
 for each regular cardinal $\kappa$ and each increasing
sequence $\{\mathcal{U}_{\alpha}: \alpha<\rho \}$ of family of
open subsets of $X$ such that $\bigcup_{\alpha<\kappa}
\mathcal{U}_{\alpha}$ is a $\omega$-cover of $X$ there is a
$\beta<\kappa$ so that $\mathcal{U}_{\beta}$ is a $\omega$-cover
of $X$. Thus $(C(X),\tau_{p},\tau_{k})$ has property that for each
regular cardinal $\kappa$ and each increasing sequence
$\{\mathcal{U}_{\alpha}: \alpha<\rho \}$ of family of cozero
subsets of $X$ such that $\bigcup_{\alpha<\kappa}
\mathcal{U}_{\alpha}$ is a $k$-cover of $X$ there is a
$\beta<\kappa$ so that $\mathcal{U}_{\beta}$ is a $\omega$-cover
of $X$. Consider  a set

$F_{\beta}=C(X)\setminus \{h\in C(X): |h(x)|<\frac{1}{2}$ for $
x\leq \beta$ and $|h(x)|<1$ for $x> \beta \}$

\noindent
where $\beta$ is a limit ordinal in $X$. Note that $F_{\beta}$ is
closed set in $C_{k}(X)$, but it is not closed set in $C_{p}(X)$.
So $F=\bigcup_{\beta\leq\omega_1} F_{\beta}=F_{\omega_1}$. For
$\textbf{0}$ we have $\textbf{0}\in cl_{\tau_p} F \setminus F$. It
follows that bitopological space $(C(X),\tau_{p},\tau_{k})$ has
not countable $(\tau_{k},\tau_{p})$-$T$-tightness.

\medskip

\begin{theorem} Let $X$ be a Tychonoff space,
$\lambda\in \Psi$. Then the following are equivalent:

\begin{enumerate}

\item  $C_{\lambda}(X)$ has countable $T$-tightness;

\item  for each regular cardinal $\rho$ and each increasing
sequence $(\mathcal{U}_{\alpha}: \alpha<\rho \}$ of family of
cozero subsets of $X$ such that $\bigcup_{\alpha<\rho}
\mathcal{U}_{\alpha}$ is a $\lambda$-$f$-cover of $X$ there is a
$\beta<\rho$ so that $\mathcal{U}_{\beta}$ is a
$\lambda$-$f$-cover of $X$.

\end{enumerate}

\end{theorem}

\begin{proof}

$(1)\Rightarrow(2)$. This follows from  Theorem \ref{th1}.

$(2)\Rightarrow(1)$. Conversely, suppose  that for each regular
cardinal $\rho$ and each increasing sequence
$(\mathcal{U}_{\alpha}: \alpha<\rho )$ of family of cozero subsets
of $X$ such that $\bigcup_{\alpha<\rho} \mathcal{U}_{\alpha}$ is a
$\lambda$-$f$-cover of $X$ there is a $\beta<\rho$ so that
$\mathcal{U}_{\beta}$ is a $\lambda$-$f$-cover of $X$.  Let
$\kappa$ be uncountable regular cardinal and $(A_{\alpha}:
\alpha<\kappa)$ be increasing sequence
 of closed subsets of $C_{\lambda}(X)$.

 Suppose $g\in \overline{\bigcup_{\alpha<\kappa} A_{\alpha}}$.
Since $C_{\lambda}(X)$ is homogeneous, we may assume $g=f_0$ where
$f_0:=\textbf{0}$. For every $n\in \mathbb{N}$ and
$\alpha<\kappa$, we set

$\mathcal{U}_{n,\alpha}=\{f^{-1}(W_{n}): f\in A_{\alpha}\}$, where
$W_{n}=(-\frac{1}{n},\frac{1}{n})$, and
$\mathcal{U}_{n}=\bigcup_{\alpha<\kappa} \mathcal{U}_{n,\alpha}$.
Every $\mathcal{U}_{n}$ is an $\lambda$-cover of $X$.

Indeed, let $F\in \lambda$ and consider the neighborhood
$[F,W_{n}]$ of $f_0$. By $f_{0}\in
\overline{\bigcup_{\alpha<\kappa} A_{\alpha}}$, there exist
$\alpha<\kappa$ and $f\in A_{\alpha}\bigcap [F, W_{n}]$. Then
$F\subset f^{-1}(W_{n})\in \mathcal{U}_{n,\alpha}$. Thus, every
$\mathcal{U}_{n}$ is an $\lambda$-cover of $X$. For every $n\in
\mathbb{N}$, we can find $\alpha_n <\kappa$ such that
$\mathcal{U}_{n,\alpha_{n}}$ is an $\lambda$-cover of $X$. Let
$\gamma$  be the supremum of $\alpha_n$'s. Then for every $n\in
\mathbb{N}$, $\mathcal{U}_{n,\gamma}$ is an $\lambda$-cover of
$X$. Now we claim $f_{0}\in A_{\gamma}$. Let $[F,W]$ be any
neighborhood of $f_{0}$ and choose $n\in \mathbb{N}$ with
$W_n\subset W$. Since $\mathcal{U}_{n,\gamma}$
 is an $\lambda$-cover of $X$, there exists an $f\in A_{\gamma}$
 such that $F\subset f^{-1}(W_{n})$. Then $f\in A_{\gamma}\bigcap [F,W]$.
Thus $f_{0}\in \overline{A_{\gamma}}=A_{\gamma}$. We conclude that
$\bigcup_{\alpha<\kappa} A_{\alpha}$ is closed in
$C_{\lambda}(X)$.

\end{proof}

\section{Bitopological R-separability and M-separability}

 $\cal R$-separability and $\cal M$-separability in bitopological
 spaces were first introduced and studied in ~{\cite{kooz}}. We have some
 analogous results on $\cal R$-separability and $\cal
 M$-separability of bitopological functional spaces endowed with
 set open topology.

\begin{definition} Let $X$ be a topological space and $\lambda$ be
a family subsets of $X$. Then the space $X$ said to be separably
$\lambda$-submetrizable if there are separable metric space $Y$,
continuous map $f$ the space $X$ onto $Y$ such that $f$ is
one-to-one on $\bigcup \lambda$.
\end{definition}

\noindent
Note that if $\bigcup \lambda=X$ then $X$ is a submetrizable
space.

 For example we consider the Mrowka-Isbell space. Let
$\mathcal{M}$ be a maximal infinite family of infinite subsets of
$\mathbb{N}$ such that the intersection of any two members of
$\mathcal{M}$ is finite, and let $\Psi= \mathbb{N}\bigcup
\mathcal{M}$, where a subset $U$ of $\Psi$ is defined to be open
provided that for any set $M\in \mathcal{M}$, if $M\in U$ then
there is a finite subset $F$ of $M$ such that $\{M\}\bigcup
M\setminus F \subset U$.

 So the Mrowka-Isbell space is separably
$\lambda$-submetrizable space (where $\lambda$ is a family finite
subsets of set $\mathbb{N}$ of isolate points of space $\Psi$),
but it is not submetrizable space.

\begin{theorem}
\label{th5.2} Let $X$ be a Tychonoff space, $\lambda\in \Psi$.
Then following conditions are equivalent:

\begin{enumerate}

\item $C_{\lambda}(X)$ is a separable space;

\item $X$ is a separably $\lambda$-submetrizable space.

\end{enumerate}

\end{theorem}

\begin{proof}
$(1) \Rightarrow (2)$. Let $D$ be a countable dense subset of
$C(X)$. Let us observe that for the diagonal map
$f=\triangle_{f_{i}\in D} f_{i}$ we have $f: X \mapsto Y$ where
$Y=f(X)\subseteq \prod\limits_{i\in \mathbb N} \mathbb R_{i}$ and
 $f$ is one-to-one on $\bigcup
\lambda$.

$(2) \Rightarrow (1)$. Consider the set $C(f(X))$ with
$f(\lambda)$-open topology. Note that $f(\lambda)$ is the family
of compact subsets of $f(X)$ and it is closed under compact
subsets $f(X)$ of its elements. By Theorem (N.Noble, \cite{nb}),
the space $C_{c}(f(X))$ is separable space then
$C_{f(\lambda)}(f(X))$ is a separable space. It follows
immediately that $C_{\lambda}(X)$ is a separable space.

\end{proof}

\begin{corollary}
 Let $X$ be a Tychonoff space, $\lambda\in
\Psi$ and let $C_{\lambda}(X)$ be a separable space. Then any
element of family $\lambda$ is a metrizable compact subset of $X$.
\end{corollary}

\begin{theorem}
\label{th2}

Let $X$ be a Tychonoff separably $\lambda$-submetrizable space,
$\lambda$, $\mu \in \Psi$ and $\mu\subseteq \lambda$. Then the
following are equivalent:

\begin{enumerate}

\item $X\in S_{1}(\Lambda(\lambda),\Lambda(\mu))$;

\item $(C(X),\tau_{\mu},\tau_{\lambda})$ is
$R_{(\tau_{\lambda},\tau_{\mu})}$ - separable.

\end{enumerate}

\end{theorem}

\begin{proof}

$(1) \Rightarrow (2)$. By Theorem \ref{th5.2}, $C_{\lambda}(X)$ is
separable. On the other hand, by Theorem \ref{th4.9}, the bispace
$(C(X),\tau_{\lambda},\tau_{\mu})$ has countable
$(\tau_{\lambda},\tau_{\mu})$- strong fan tightness if (and only
if) $X\in S_{1}(\Lambda(\lambda),\Lambda(\mu))$. By Corollary 3 in
~{\cite{kooz}}, we obtain $(C(X),\tau_{\mu},\tau_{\lambda})$ is
$R_{(\tau_{\lambda},\tau_{\mu})}$ - separable.

$(2) \Rightarrow (1)$. Let $(\mathcal{U}_{n}: n\in \mathbb{N})$ be
a sequence of $\lambda$-covers of $X$. For every $n\in \mathbb{N}$
let  $A_{n}=\{f\in C_{\lambda}(X)$:  there is $U\in
\mathcal{U}_{n}, f(X\setminus U)=\{1\} \}$.

First, we verify that each $A_{n}$ is dense in $C_{\lambda}(X)$.
To this end let us consider  $f\in C_{\lambda}(X)$ and let
$\bigcap_{i\leq m} [K_{i},V_{i}]$ be a basic neighbourhood of $f$.
The set $K=\bigcup_{i\leq m} K_{i}$ is compact and $K\in \lambda$,
and since $\mathcal{U}_{n}$ is a $\lambda$-cover, there is $U\in
\mathcal{U}_{n}$ containing $K$. There is also $g\in
C_{\lambda}(X)$ such that $g(X\setminus U)=\{1\}$ and
$g\upharpoonright K= f\upharpoonright K$. Then $g\in
\bigcap_{i\leq m} [K_{i},V_{i}]\bigcap A_{n}$.

Since $(C(X),\tau_{\mu},\tau_{\lambda})$ is
$R_{(\tau_{\lambda},\tau_{\mu})}$ - separable there are functions
$f_{n}\in A_{n}$, $n\in \mathbb{N}$, such that the set $\{f_{n}:
n\in \mathbb{N}\}$ is dense in $C_{\mu}(X)$. Let $U_{n}\in
\mathcal{U}_{n}$ be a set for which $f_{n}(X\setminus
U_{n})=\{1\}$ holds. We claim that $\{U_{n}: n\in \mathbb{N}\}\in
\Lambda(\mu)$. Let $F\in \mu$. Suppose that  there is a point
$x_{n}\in F\setminus U_{n}$ for each $n\in \mathbb N$ which means
 $f_{n}(x_{n})=1$ and it contradicts the fact that  $\{f_{n}: n\in \mathbb{N}\}$ is dense in
$C_{\mu}(X)$.

\end{proof}

\begin{corollary}
\label{th20}

Let $X$ be a Tychonoff separably $\lambda$-submetrizable space,
$\lambda \in \Psi$. Then the following are equivalent:

\begin{enumerate}

\item $X\in S_{1}(\Lambda(\lambda),\Lambda(\lambda))$;

\item $C_{\lambda}(X)$ is $R$ - separable.

\end{enumerate}

\end{corollary}

\begin{theorem}
Let $X$ be a Tychonoff separably $\lambda$-submetrizable space,
$\lambda$, $\mu \in \Psi$ and $\mu\subseteq \lambda$. Then the
following are equivalent:

\begin{enumerate}

\item $X\in S_{fin}(\Lambda(\lambda),\Lambda(\mu))$;

\item $(C(X),\tau_{\mu},\tau_{\lambda})$ is
$M_{(\tau_{\lambda},\tau_{\mu})}$ - separable.

\end{enumerate}

\end{theorem}

\begin{proof}

$(1) \Rightarrow (2)$. The space $C_{\lambda}(X)$ is separable
since $X$ is separably $\lambda$-submetrizable space. On the other
hand, by Theorem \ref{th1}, $(\tau_{\lambda},\tau_{\mu})$-fan
tightness of $(C(X),\tau_{\mu},\tau_{\lambda})$ is countably if
and only if $X$ has selection property
$S_{fin}(\Lambda(\lambda),\Lambda(\mu))$ and now apply  Corollary
6 in ~{\cite{kooz}}.

$(2) \Rightarrow (1)$. Assume that  $(\mathcal{U}_{n}:n\in \mathbb{N})$ be
a sequence of $\lambda$-covers. For every $n\in \mathbb{N}$ we set

$V_{n}=\{f\in C_{\lambda}(X)$: there is $U\in \mathcal{U},
f(X\setminus U)=\{1\}\}$.

We follow the  proof of Theorem \ref{th2} to show that each $V_{n}$
is dense in $C_{\lambda}(X)$.

By the hypothesis   $(C(X),\tau_{\mu},\tau_{\lambda})$ is
$M_{(\tau_{\lambda},\tau_{\mu})}$ - separable there are finite
sets $W_{n}=\{f_{n,1},..., f_{n,m_{n}}\}\subset V_{n}$, $n\in
\mathbb{N}$, such that the set $\bigcup_{n\in \mathbb{N}} W_{n}$
is dense in $C_{\mu}(X)$.

Now consider the set
$\mathbb{D}_{n}=\{U_{n,1},..., U_{n,m_{n}}: f_{n,i}(X\setminus
U_{n,i})=\{1\}, i\leq m_{n}\}$ which is a finite subset  of $\mathcal{U}_{n}$. It remains to show that $\bigcup_{n\in
\mathbb{N}} \mathbb{D}_{n}\in \Lambda(\mu)$. Let $F\in \mu$. For some $j\in  \mathbb{N}$ we have
$[F,(-1,1)] \cap W_{j}$ , i.e. there is a function $f_{j,m_{j}}\in W_{j}$ such
that $f_{j,m_{j}}(x)\in (-1,1)$ for each $x\in F$. This means
$F\subset U_{j,m_{j}}$ as required.

\end{proof}

\begin{corollary}\label{cor1}
Let $X$ be a Tychonoff separably $\lambda$-submetrizable space,
$\lambda \in \Psi$. Then the following are equivalent:

\begin{enumerate}

\item $X\in S_{fin}(\Lambda(\lambda),\Lambda(\lambda))$;

\item $C_{\lambda}(X)$ is $M$ - separable.

\end{enumerate}

\end{corollary}

Recall that a bispace $(X,\tau_1,\tau_2)$ is
$(\tau_i,\tau_j)$-Pytkeev $(i\neq j; i,j=1,2)$ ~{\cite{ko1}} (see
also ~{\cite{panpa}}) if for each $A\subset X$ and each $x\in
Cl_i(A)\setminus A$ there are infinite sets $B_n\subset A$, $n\in
\mathbb{N}$, such that each $\tau_j$-neighbourhood of $x$ contains
some $B_n$.

By Theorem 9 in ~{\cite{kooz}}, we have following result.

\begin{theorem}
Let $X$ be a Tychonoff separably $\lambda$-submetrizable space,
$\lambda$, $\mu \in \Psi$ and $\mu\subseteq \lambda$. If
$(C(X),\tau_{\mu},\tau_{\lambda})$ is
$M_{(\tau_{\mu},\tau_{\lambda})}$ - separable and
$(\tau_{\mu},\tau_{\lambda})$-Pytkeev bispace, then it is
$R_{(\tau_{\mu},\tau_{\lambda})}$ - separable.

\end{theorem}

\begin{corollary}
Let $X$ be a Tychonoff separably $\lambda$-submetrizable space,
$\lambda \in \Psi$. If $C_{\lambda}(X)$ is $M$ - separable and
Pytkeev space, then it is $R$ - separable.

\end{corollary}

\section{Bitopological $H$-separability and $GN$-separability}

In this section we will be interested in some results on
bitopological $H$-separability and $GN$-separability.

We begin by recalling the notion of weak selectively Reznichenko property
for the bitopological spaces.
A bispace $(X,\tau_1, \tau_2)$ has the weak selectively
$(\tau_{i},\tau_{j})$-Reznichenko property ($i\neq j; i,j=1,2$),
if for each sequence $(A_{n}:n\in \mathbb{N})$ of subsets of $X$
and each point $x\in \bigcap_{n\in \mathbb{N}} Cl_{i}(A_n)$ there
are finite sets $B_n\subset A_n$, $n\in \mathbb{N}$, such that
each $\tau_{j}$-neighbourhood of $x$ intersects $B_{n}$ for all
but finitely many $n$.

The definition of the selective bitopological version of the
Reznichenko property was  given in  {\cite{panpa}}. It has been
characterized by considering the compact-open and the topology of
pointwise convergence on the set of all continuous real-valued
functions in ~{\cite{pav}}.

The notion of weak selectively Reznichenko property
was introduced in ~{\cite{kooz}}. The following results may be proved in much the same way as
Theorem 10 and Theorem 11 in~{\cite{kooz}}.

\begin{theorem}\label{th5}
Let $X$ be a Tychonoff separably $\lambda$-submetrizable space,
$\lambda$, $\mu \in \Psi$ and $\mu\subseteq \lambda$. Then the
following are equivalent:

\begin{enumerate}

\item $(C(X),\tau_{\mu},\tau_{\lambda})$ is
$H_{(\tau_{\lambda},\tau_{\mu})}$ - separable;

\item For each sequence $(\mathcal{U}_{n}:n\in \mathbb{N})$ of
$\lambda$-covers there is a sequence $(\mathcal{V}_{n}:n\in
\mathbb{N})$ of finite sets such that for each $n$,
$\mathcal{V}_{n}\subset \mathcal{U}_{n}$ and each $F\in \mu$ is
contained in an element of $\mathcal{V}_{n}$ for all but finitely
many $n\in \mathbb{N}$.

\end{enumerate}

\end{theorem}

\begin{proof} $(1) \Rightarrow (2)$. Let  $(\mathcal{U}_{n}:n\in \mathbb{N})$ be
a sequence of $\lambda$-covers. For every $n\in \mathbb{N}$ let

$D_{n}=\{f\in C_{\lambda}(X)$: there is $U\in \mathcal{U}_n,
f(X\setminus U)=\{1\}\}$.

\noindent
Easily one can prove that each  $D_{n}$ is dense in
$(C(X),\tau_{\lambda})$. Since $(C(X),\tau_{\mu},\tau_{\lambda})$
is $H_{(\tau_{\lambda},\tau_{\mu})}$ - separable there are finite
sets $F_{n}\subset D_{n}$, $n\in \mathbb{N}$, such that each
$\tau_{\mu}$-open set intersects $F_n$ for all but finitely many
$n$. Let $\mathcal{V}_{n}$, $n\in \mathbb{N}$, be the family of
sets $U_f \in \mathcal{U}_n$, $f\in F_{n}$, such that
$f(X\setminus U_f)=\{1\}$. We need to verify that  the sequence
$(\mathcal{V}_{n}: n\in \mathbb{N})$ witnesses that $X$ satisfies
(2).

Let $K\in \mu$. The open neighborhood  $H=[K,(-1,1)]$ intersects
$F_m$ for each $m$ bigger than  $m_0$; now pick $f_m\in H\bigcap F_m$,
$m>m_0$. Then for each $m>m_0$, $K\subset U_{f_m}\in
\mathcal{V}_{m}$, as required in (2).

\medskip
$(2) \Rightarrow (1)$. The proof consists of two parts.

\smallskip
$\bf Claim$ 1. $C(X)$ has the weak selectively
$(\tau_{\lambda},\tau_{\mu})$-Reznichenko property.

We take a sequence  $(D_n: n\in \mathbb{N})$  of subsets of $C(X)$
with the $\tau_{\lambda}$-closures of which contain $\bf{0}$. For
every $T\in \lambda$ and every $m\in \mathbb{N}$ the
$\tau_{\lambda}$-neighborhood $[T,\frac{1}{m}]$ of $\bf{0}$
intersects each $D_{n}$. So for each $n\in \mathbb{N}$ there
exists a function $f_{T,n,m}\in D_{n}$ satisfying
$|f_{T,n,m}(x)|<\frac{1}{m}$ for each $x\in T$. For each $n$ set

$\mathcal{U}_{n,m}=\{f^{-1}(-\frac{1}{m}, \frac{1}{m}): m\in
\mathbb{N}, f\in D_{n}\}$.

(For a bijection $\varphi : \mathbb{N}^{2}\mapsto \mathbb{N}$ we
put   $\mathcal{U}_{n,m}:= \mathcal U_{\varphi(m,n)}$). We claim
that for each $n,m\in \mathbb{N}$, each $C\in \lambda$ is
contained in an element of $\mathcal{U}_{n,m}$. Indeed, if $C\in
\lambda$, then there is $f_{C,n,m}\in [C,\frac{1}{m}]\bigcap
D_{n}$. Hence $|f_{C,n,m}(x)|< \frac{1}{m}$ for each $x\in C$.
This shows that $C\subset
f^{-1}_{C,n,m}(-\frac{1}{m},\frac{1}{m})\in \mathcal{U}_{n,m}$.

Put $S:=\{m\in \mathbb{N}: X\in \mathcal{U}_{n,m}$ for some $n\in
\mathbb{N}\}$. There are two cases to consider.

$\bf Case$ 1. $S$ is infinite.

There are $m_1<m_2<...$ in $M$ and (the corresponding) $n_1, n_2,
...$ in $\mathbb{N}$ such that
$f^{-1}_{T_i,n_i,m_i}(-\frac{1}{m_i}, \frac{1}{m_i})=X$ for all
$i\in \mathbb{N}$ and some $T_i\in \lambda$. Let $[R,\epsilon]$ be
a $\tau_{\mu}$- neighborhood of $\bf{0}$. Pick $m_k$ such that
$\frac{1}{m_k}<\epsilon$. For every $m_i>m_k$ we have
$f_{T_i,n_i,m_i}(x)\in (-\frac{1}{m_i}, \frac{1}{m_i})$ for each
$x\in X$ and so $f_{T_i,n_i,m_i}\in [R, \frac{1}{m_i}]\subset
[R,\epsilon]$. This means that the sequence $(f_{T_i,n_i,m_i}:
i\in \mathbb{N})$ $\tau_{\mu}$-converges to $\bf{0}$, hence $C(X)$
has the weak selectively $(\tau_{\lambda},
\tau_{\mu})$-Reznichenko property at $\bf{0}$.

$\bf Case$ 2. Consider the case $S$ is finite.

There is $m_0\in \mathbb{N}$ such that for each $m\geq m_0$ and
each $n\in \mathbb{N}$, the set $\mathcal{U}_{n,m}$ is a
$\lambda$-cover of $X$. We may suppose $m_0=1$. Further, we can
consider only $\lambda$-covers $\mathcal{U}_{n,n}$, $n\in
\mathbb{N}$. We can apply the condition (2) of this theorem to the
sequence $\mathcal{U}_{n,m}$ to get a sequence $\mathcal{V}_{n,m}$
where for each $n\in \mathbb{V}$ $\mathcal{V}_{n,m}$ is a finite
subset of $\mathcal{U}_{n,m}$ so that each $R\in \mu$ belongs to
some $V\in \mathcal{V}_{n,n}$ for all but finitely many $n$.
Choose the corresponding functions
$f_{T_V,\frac{1}{n},\frac{1}{n}}$, $V\in \mathcal{V}_{n,n}$, and
put $F_{n}=\{ f_{T_V,\frac{1}{n},\frac{1}{n}}: V\in
\mathcal{V}_{n,n}\}$. Then each $F_n$ is a finite subset of $D_n$.
Let $[R,\frac{1}{i}]$ be a neighborhood of $\bf{0}$. Let $n_0$ be
such that $\frac{1}{n}<\frac{1}{i}$ and for each $n>n_0$ there is
$V_{n}\in \mathcal{V}_{n,n}$ containing $R$. Choose a
corresponding $f_n\in F_n$. Since this can be done for all
$n>n_0$, we conclude that for all $n>n_0$ we have $f_n\in
[R,\frac{1}{i}]$, i.e., $F_n\bigcap [R,\frac{1}{i}]\neq \emptyset$
for all $n>n_0$.

\smallskip
We now get back to proving the theorem.

 $\bf Claim$ 2. $C(X)$ is
$H_({\tau_{\lambda},\tau_{\mu}})$-separable.

Since  $X$ is a Tychonoff separably $\lambda$-submetrizable space,
and $\tau_{\mu}\leq \tau_{\lambda}$, there is a countable dense
subset $D=\{d_n: n\in \mathbb{N}\}$ in $(C(X), \tau_{\lambda})$ so
also in $(C(X), \tau_{\mu})$. Let $(E_{n}: n\in \mathbb{N})$ be a
sequence of dense subsets of $(C(X),\tau_{\lambda})$. Fix $m\in
\mathbb{N}$. Since $d_{m}\in Cl_{\tau_{\lambda}}(E_n)$ for each
$n\in \mathbb{N}$, and $C(X)$ has the weak selectively
$(\tau_{\lambda}, \tau_{\mu})$-Reznichenko property, there are
finite sets $R_{n,m}$, such that for each $n$, $R_{n,m}\subset
E_n$ and each $\tau_{\mu}$-neighborhood of $d_m$ intersects all
but finitely many $R_{n,m}$. For each $n$ put
$R_{n}=\bigcup\{R_{n,m}: m\leq n\}.$ The sequence $(R_n: n\in
\mathbb{N})$ witnesses for $(E_{n}: n\in \mathbb{N})$ that $C(X)$
is $\cal H_{(\tau_{\lambda},\tau_{\mu})}$-separable. Indeed, let
$G$ be an open set in $(C(X),\tau_{\mu})$. Then there is $d_m\in
G$, hence $G$ meets all but finitely many $R_n$.

\end{proof}

\begin{corollary}\label{th50}
Let $X$ be a Tychonoff separably $\lambda$-submetrizable space,
$\lambda \in \Psi$. Then the following are equivalent:

\begin{enumerate}

\item $C_{\lambda}(X)$ is $H$ - separable;

\item For each sequence $(\mathcal{U}_{n}:n\in \mathbb{N})$ of
$\lambda$-covers there is a sequence $(\mathcal{V}_{n}:n\in
\mathbb{N})$ of finite sets such that for each $n$,
$\mathcal{V}_{n}\subset \mathcal{U}_{n}$ and each $F\in \lambda$
is contained in an element of $\mathcal{V}_{n}$ for all but
finitely many $n\in \mathbb{N}$.

\end{enumerate}

\end{corollary}

\begin{theorem}\label{th6.2}
If $(C(X),\tau_{\mu},\tau_{\lambda})$ is $GN_{(\tau_{\lambda},\tau_{\mu})}$ - separable, then $X$ satisfies
$S_{fin}(\Lambda(\lambda),\Lambda(\mu)^{gp})$.
\end{theorem}

\begin{proof}
Let $(\mathcal{U}_{n}: n\in \mathbb{N})$ be a sequence of
$\lambda$-covers of $X$. Now  we define sets $D_n$ in
$(C(X),\tau_{\lambda})$ as in the proof of Theorem $\ref{th5}$.
These sets are  $d$-dense in $C(X)$.  Apply the fact that  $C(X)$
is $\cal GN_{(\tau_{\lambda},\tau_{\mu})}$-separable, there are
$f_n\in D_n$, $n\in \mathbb{N}$, such that $D=\{f_n: n\in
\mathbb{N}\}$ is $\tau_{\mu}$-groupable, i.e. $D=\bigcup_{m\in
\mathbb{N}} G_{m}$, where each $G_m=\{f^{k_1}_{m},...,
f^{k_m}_{m}\}$ is a finite subset of $D$ and each
$\tau_{\mu}$-open set meets all but finitely many $G_m$. For each
$m\in \mathbb{N}$, let

$\mathcal{V}_m=\{U^{k_i}_m : f^{k_i}_m(X\setminus
U^{k_i}_m)=\{1\}, i\leq m \}$.

Let us show that  each $F\in \mu$ is contained in some $V\in
\mathcal{V}_m$ for all but finitely many $m$. Let $F\in \mu$. Then
the $\tau_{\mu}$-open set $[F,1]$ intersects $G_m$ for all $m$
bigger than some $m_0\in \mathbb{N}$. Let $f^{k_j}_{m}\in
[F,1]\bigcap G_m$, $m\geq m_0$. Then $F\subseteq U^{k_j}_{m}$,
$m\geq m_0$. It shows that  $X$ has
$S_{fin}((\Lambda(\lambda),\Lambda(\mu)^{gp})$.

\end{proof}

\begin{corollary}\label{th6.3}
If $C_{\lambda}(X)$ is $GN$ - separable, then $X$ satisfies
$S_{fin}(\Lambda(\lambda),\Lambda(\lambda)^{gp})$.
\end{corollary}

From Theorem $\ref{th4.9}$, Theorem $\ref{th6.2}$ and Corollary
$\ref{th6.3}$ we obtain

\begin{theorem}
If $(C(X),\tau_{\mu},\tau_{\lambda})$ is $GN_{(\tau_{\lambda},\tau_{\mu})}$ - separable, then $X$ satisfies
$S_{1}(\Lambda(\lambda),\Lambda(\mu))$ and
$S_{fin}(\Lambda(\lambda),\Lambda(\mu)^{gp})$.
\end{theorem}

\begin{corollary}
If $(C(X),\tau_{\mu},\tau_{\lambda})$ is $GN_{(\tau_{\lambda},\tau_{\mu})}$ - separable, then
$(C(X),\tau_{\mu},\tau_{\lambda})$ is
$R_{(\tau_{\lambda},\tau_{\mu})}$ - separable as well as
$H_{(\tau_{\lambda},\tau_{\mu})}$ - separable.
\end{corollary}

\begin{corollary}
If $C_{\lambda}(X)$ is $GN$ - separable, then
$C_{\lambda}(X)$ is $R$ - separable as well as $H$ - separable.
\end{corollary}

\section{Examples}

We consider some examples which separating different type of
separability in function bispaces.

\begin{example} Let $\mathbb{I}=[0,1]\subset \mathbb{R}$.

\smallskip
$\bullet$ By Example 2.14 in ~{\cite{bbmt}}, $C_p(\mathbb{I})$ is
$M$-separable i.e. $\mathbb{I}$ has the property
$S_{fin}(\Lambda(p),\Lambda(p))$ by Corollary \ref{cor1}.

\smallskip
Since each $k$-cover of $\mathbb{I}$ is $\omega$-cover  we have
that $\mathbb{I}\in$ $S_{fin}(\Lambda(k),\Lambda(p))$. Hence the
space $(C(\mathbb{I}),\tau_p, \tau_k)$ is
$M_{(\tau_{k},\tau_{p})}$ - separable.

\smallskip
$\bullet$  By Proposition 61 in ~{\cite{bbm}}, $C_p(\mathbb{I})$
is not  $R$-separable, and, by Fact 2.1 in ~{\cite{kooz}}, space
$(C(\mathbb{I}),\tau_p, \tau_k)$ is not $R_{(\tau_{k},\tau_{p})}$ -
separable.

\smallskip
It follows that the bitopological space $(C(\mathbb{I}),\tau_p,
\tau_k)$ is $M_{(\tau_{k},\tau_{p})}$ - separable, but it is not
$R_{(\tau_{k},\tau_{p})}$ - separable.
\end{example}

\begin{example} Let $X=\omega^{\omega}$. Then $X$ is
$p$-Lindel$\ddot{o}$f space, but $C_p(X)$ is not  $M$-separable.
\end{example}
Recall that $\mathfrak{b}$ denote the minimum of cardinality of an
unbounded set in $\omega^{\omega}$~{\cite{do}}.

\begin{example} Let $X$ be an uncountable, second countable space of cardinality less than $\mathfrak{b}$. By Corollary 43 in ~{\cite{bbm}},
$C_p(X)$ is $H$-separable, but $C_p(X)$ is not $R$-separable,
and, hence, $C_p(X)$ is not  $GN$-separable.
\end{example}

\noindent
{\bf Question 6.4}. Does there exist an $X$ such that
$(C(X),\tau_{\mu},\tau_{\lambda})$ is
$R_{(\tau_{\lambda},\tau_{\mu})}$ - separable and $H_{(\tau_{\lambda},\tau_{\mu})}$ - separable, but it isn't $GN_{(\tau_{\lambda},\tau_{\mu})}$- separable for some $\lambda$,
$\mu \in \Psi$ and $\mu\subseteq \lambda$ ( for $\lambda=k$ and
$\mu=p$ )?

\bibliographystyle{model1a-num-names}
\bibliography{<your-bib-database>}

\begin {thebibliography}{}

\bibitem{ad}
R. Arens, J. Dugundji, \textit{Topologies for  function spaces},
Pasific J.Math. 1 (1951), 5-31.

\bibitem{arh}
A.V. Arhangel'skii, \textit{Topological function spaces}, Kluwer
Academic Publishers, (1992).


\bibitem{bbm}
A. Bella, M. Bonanzinga, M. Matveev, \textit{Variations of
selective separability}, Topology Appl. 156 (2009), 1241-1252.


\bibitem{bbmt}
A. Bella, M. Bonanzinga, M. Matveev, V. Tkachuk, \textit{Selective
separability: general facts and behavior in countable spaces},
Topology Proc. 32 (2008), 15--30.


\bibitem{bbms}
A. Bella, M. Matveev, S. Spadaro, \textit{ Variations of selective
separability II: Discrete sets and the influence of convergence
and maximality}, Topology Appl. 159 (1) (2012), 253--271.

\bibitem{tsaban2}
M. Bonanzinga, F. Cammaroto, B. Pansera, B. Tsaban,
 \textit{Diagonalizations of dense families}, Topology Appl. 165 (2014), 12--25.

\bibitem{do}
E.K. van Douwen, \textit{Integers in topology, in: K. Kunen,
J.E.Vaughan (Eds.). Handbook of Set-Theoretic Topology}, Elsevier
Sci. Pub. B.V. (1984), 111--168.


\bibitem{enge}
R. Engelking, \textit{General Topology}, PWN, Warsaw, (1977); Mir,
Moscow, (1986).

\bibitem{mkm}
G. Di Maio, Lj.D.R. Ko$\check{c}$inac, E. Meccairiello,
\textit{Selection principles and hyperspce topologies}, Topology
Appl. 153 (2005), 912--923.

\bibitem{gs}
G. Grunhage, M.Sakai, \textit{Selective separability and its
variations}, Topology Appl. 158 (12) (2011), 1352--1359.

\bibitem {hur}
W. Hurewicz, \textit{ \"Uber die Verallgemeinerung des Borelschen}
Theorems, Math. Z. 24 (1925), 401-–425.

\bibitem{ko}
Lj.D.R. Ko$\check{c}$inac, \textit{Closure properties of function spaces},
Appl.Gen.Top. 4 (2003), 255--261.

\bibitem{ko1}
Lj.D.R. Ko$\check{c}$inac, \textit{Selected results on selection
principles},In: Proc. Third Seminar Geom. Topol. (July 15-17,
2004,Tabriz, Iran), 105--109.

\bibitem{kooz}
Lj.D.R. Ko$\check{c}$inac, S. $\ddot{O}$z\c{c}a$\breve{g}$,
\textit{Versions of separability in bitopological spaces},
Topology Appl. 158 (2011), 1471--1477.

\bibitem{kooz1}
Lj.D.R. Ko$\check{c}$inac, S. $\ddot{O}$z\c{c}a$\breve{g}$,
\textit{Bitopological spaces and selection principles}, Cambridge
Scientific Publishers (2012), 243--255.

\bibitem{menger}
K. Menger, \textit{Einige \"Uberdeckungss\"atze der
Punktmengenlehre}, Sitzungsberichte Abt. 2a, Mathematik,
Astronomie, Physik, Meteorologie und Mechanik (Wiener Akademie,
Wien) 133 (1924), 421--444.

\bibitem{nb}
N.Noble, \textit{The density character of function spaces}, Proc.
Amer. Math. Soc. 42:1 (1974), 228--233.

\bibitem{uch}
 I. Juh$\acute{a}$sz, \textit{ Variations on tightness}, Studia Sci. Math. Hangar.
24 (1989), 179–-186.

\bibitem{os2}
A.V. Osipov, \textit{ Topological-algebraic properties of function
spaces with set-open topologies}, Topology Appl. 159(3) (2012),
800-805.

\bibitem{os1}
A.V. Osipov, \textit{ The Set-Open topology}, Top. Proc. 37
(2011), 205-217.

\bibitem{os3}
A.V. Osipov, \textit{On the completeness properties of the
$C$-compact-open topology on $C(X)$}, Ural Math. Journal, Vol. 1
(2015), 61-67.


\bibitem{panpa}
B.A. Pansera, V. Pavlovi$\acute{c}$, \textit{Open covers and
function spaces}, Math.Vesnik  58 (2006), 57--70.

\bibitem{pav}
V. Pavlovi$\acute{c}$, \textit{A selective bitopological version
of the Reznichenko property in function spaces}, Topology Appl.
156 (2009), 1636--1645.

\bibitem {roth}
F. Rothberger, \textit{ Eine Verscharfung der Eigenschaft {\sf
C}}, Fund. Math. 30 (1938), 50-55.

\bibitem{masa}
M. Sakai, \textit{Variations on tightness in function spaces},
Topology Appl. 101 (2000), 273--280.

\bibitem {sc}
M. Scheepers, \textit{Combinatorics of open covers (I): Ramsey
Theory}, Topology Appl.  69 (1996), 31--62.


\bibitem{sch}
M. Scheepers, \textit{Combinatorics of open covers (VI): Selectors
for sequences of dense sets}, Quaest. Math. 22 (1999), 109--130.

\bibitem{sc1}
M. Scheepers, \textit{Selection principles and covering properties
in topology} Note Mat. 22 (2) (2003/2004), 3--41.


\bibitem{tsaban}
B. Tsaban, \textit{Some new directions in infinite-combinatorial
topology}, In: Set Theory (J.Bagaria and S. Todorcevic, eds.)
Trends in Mathematics, Birkhauser, (2006), 225--255.


\end{thebibliography}







\end{document}